\documentclass[12pt,reqno]{article}

\usepackage[usenames]{color}
\usepackage{amssymb}
\usepackage{amsmath}
\usepackage{amsthm}
\usepackage{amsfonts}
\usepackage{amscd}
\usepackage{graphicx}

\usepackage[colorlinks=true,
linkcolor=webgreen,
filecolor=webbrown,
citecolor=webgreen]{hyperref}

\definecolor{webgreen}{rgb}{0,.5,0}
\definecolor{webbrown}{rgb}{.6,0,0}

\usepackage{color}
\usepackage{fullpage}
\usepackage{float}

\usepackage{graphics}
\usepackage{latexsym}

\begin{document}

\theoremstyle{plain}
\newtheorem{theorem}{Theorem}
\newtheorem{proposition}{Proposition}
\newtheorem{corollary}[theorem]{Corollary}
\newtheorem{lemma}{Lemma}
\newtheorem*{example}{Examples}
\newtheorem{remark}{Remark}

\begin{center}
\vskip 1cm{\LARGE\bf 
Some notes on a Fibonacci--Lucas identity \\
}
\vskip 1.1cm
{\large
Kunle Adegoke \\
Department of Physics and Engineering Physics \\ Obafemi Awolowo University\\Ile-Ife, Nigeria \\
\href{mailto:adegoke00@gmail.com}{\tt adegoke00@gmail.com}

\vskip 0.25 in

Robert Frontczak\\
Independent Researcher \\
Reutlingen,   Germany\\
\href{mailto:robert.frontczak@web.de}{\tt robert.frontczak@web.de}

\vskip 0.25 in

Taras Goy  \\
Faculty of Mathematics and Computer Science\\
Vasyl Stefanyk Precarpathian National University\\
Ivano-Frankivsk, Ukraine\\
\href{mailto:taras.goy@pnu.edu.ua}{\tt taras.goy@pnu.edu.ua}}
\end{center}

\vskip .15 in


\vskip .1 in


\medskip

\begin{abstract}
In 2016, Edgar and, independently of him, Bhatnagar sta\-ted a nice polynomial identity that connects Fibonacci and Lucas numbers. 
Shortly after their publications, this identity has been generalized in two different ways: 
Dafnis, Phillipou and Livieris provided a generalization to Fibonacci sequences of order $k$ and Abd-Elhameed and Zeyada extended Edgar--Bhatnagar identity to generalized Fibonacci and Lucas sequences. In this paper, we present more polynomial identities for generalized Lucas sequences. We discuss interesting aspects and special cases which have not been stated before but deserve recognition. Finally, we prove the polynomial analogues of these identities for Chebyshev polynomials. 
 
\end{abstract}

\section{Motivation}

As usual, the Fibonacci numbers and Lucas numbers will be denoted by $F_n$ and $L_n$, respectively. They are defined, for $n\in\mathbb Z$, through the recurrence relations 
\begin{equation*}
x_n = x_{n-1}+x_{n-2},\qquad n\ge 2,
\end{equation*}
with initial values $F_0=0$, $F_1=1$ and 
$L_0=2$, $L_1=1$. For negative subscripts we have
\begin{equation*} F_{-n}=(-1)^{n-1}F_n,\qquad L_{-n}=(-1)^n L_n.
\end{equation*}
They possess the explicit formulas (Binet forms)
\begin{equation*}
F_n = \frac{\alpha^n - \beta^n }{\alpha - \beta },\qquad L_n = \alpha^n + \beta^n,\quad n\in\mathbb Z.
\end{equation*}
Standard references on these sequences are the textbooks by Koshy \cite{Koshy} and Vajda \cite{Vajda} 
in which a huge amount of additional information is presented. 

Our motivation for these notes is the following prominent identity due to Edgar \cite{Edgar} and  Bhatnagar \cite{Bhatnagar} 
from 2016 where they generalized two Fibonacci--Lucas identities from the American Mathematical Monthly:
\begin{equation}\label{id_edgar}
\sum_{k=0}^n x^k (L_k + (x-2)F_{k+1}) = x^{n+1} F_{n+1}.
\end{equation}
This identity holds for all $x\in\mathbb{C}$ and the special cases $x=2$ and $x=3$, respectively, have been proved by
Benjamin and Quinn \cite{BenQuinn} and Marques \cite{Marques}. A polynomial variant was given by Sury \cite{Sury}. Dafnis, Philippou and Livieris have generalized this identity to Fibonacci and Lucas numbers of order $k$ and gave two different proofs 
for their generalization \cite{Dafnis1,Philippou}. In addition, using Euler's telescoping lemma Bhatnagar derived in \cite{Bhatnagar} an alternating version of \eqref{id_edgar} as follows: 
\begin{equation}\label{id_dpl}
\sum_{k=0}^n (-1)^k x^{n-k} \big(L_{k+1} + (x-2)F_{k}\big) = (-1)^{n} F_{n+1}.
\end{equation} 

On the other hand, Abd-Elhameed and Zeyada \cite{ameed_zeyada18}, without a reference to Edgar or Bhatnagar, derived two polynomial identities, one of which generalizes the polynomial identity of Sury \cite{Sury}. Finally, in a recent paper, Chung, Yao and Zhou \cite{Chung} provided extensions of Sury's relation and the alternating
Sury's relation involving Fibonacci $k$-step and Lucas $k$-step polynomials.

In this paper, we proceed in the same direction and present more polynomial identities for generalized Lucas sequences. We discuss interesting aspects and special cases which have not been stated before. Finally, we prove the polynomial analogues of these identities for Chebyshev polynomials.

\section{Additional polynomial generalizations}

In this section, we give generalizations of \eqref{id_edgar} and \eqref{id_dpl}
to the general Lucas sequences and prove two another identities of the same nature. 

For non-zero complex numbers $y$ and $z$, the Lucas sequences of the first kind, $\left(u_n(y,z)\right)_{n\geq0}$, 
and of the second kind, $\left(v_n(y,z)\right)_{n\geq0}$, are defined (see e.g., Ribenboim \cite[Chapter 1]{ribenboim})
through the recurrence relations 
\begin{equation*}
u_n(y,z) = yu_{n-1}(y,z) - zu_{n-2}(y,z), \quad n\ge 2, \quad u_0(y,z)=0,\, u_1(y,z)=1,
\end{equation*}
and
\begin{equation*}
v_n(y,z) = yv_{n-1}(y,z) - zv_{n-2}(y,z), \quad n\ge 2, \quad v_0(y,z)=2,\, v_1(y,z)=y,
\end{equation*}
with
\begin{equation*}
u_{-n}(y,z) = -u_n(y,z)z^{-n},\qquad  v_{-n}(y,z) = v_n(y,z)z^{-n}.
\end{equation*}

Denote by $\tau$ and $\sigma$ the zeros of the characteristic polynomial $x^2-yx+z$ for the Lucas sequences. Then
\begin{equation*}
\tau = \tau(y,z) = \frac{y+\sqrt{y^2-4z}}2,\qquad \sigma = \sigma(y,z) = \frac{y-\sqrt{y^2-4z}}2\,,
\end{equation*}
with $\tau + \sigma = y$, $\tau - \sigma = \sqrt{y^2-4z}$ and $\tau\sigma=z$\,.

The difference equations are solved by the Binet formulas
\begin{equation*}
u_n = \frac{\tau^n - \sigma^n}{\tau-\sigma},\qquad v_n = \tau^n+\sigma^n\,.
\end{equation*}

We need the following two lemmas, of which the first one contains basic telescoping summation identities.
\begin{lemma}
	If $(f_k)_{k\in\mathbb Z}$ is a real sequence and $j$, $n\in \mathbb{Z}$, then
	\begin{align}\label{tele_id1}
	\sum_{k = j}^n \left( {f_{k + 1} - f_k } \right) &= f_{n + 1} - f_j,\\
	\label{tele_id2}
	\sum_{k = j}^n (- 1)^{k} \left( {f_{k + 1} + f_k } \right)& = (- 1)^{n} f_{n + 1} + (- 1)^j f_j.
	\end{align}
\end{lemma}
\begin{lemma}
	For all $x$, $y$, $r\in \mathbb{C}$, $r\ne 0$ and $k\in\mathbb Z$, we have
	\begin{align}\label{eq.aklr76c}
	&\qquad \frac{x^k}{r^k} \big( {rv_k (y,z) + (xy - 2r)u_{k + 1} (y,z)} \big) = y\left( {\frac{{x^{k + 1} }}{{r^k }}u_{k + 1} (y,z) - \frac{{x^k }}{{r^{k - 1} }}u_k (y,z)} \right)\!,\\[4pt]
	\label{eq.iq7k21o}
	&\frac{{x^k }}{{r^k }} \big( {(y^2 - 4z)ru_k (y,z) + (xy - 2r)v_{k + 1} (y,z)} \big)  =  y\left( {\frac{{x^{k + 1} }}{{r^k }}v_{k + 1} (y,z) - \frac{{x^k }}{{r^{k - 1} }}v_k (y,z)} \right)\!.
	\end{align}
\end{lemma}
\begin{proof} On account of the recurrence relation
	\begin{equation}\label{eq.s4dauqs}
	u_{k+1}(y,z) + zu_{k-1}(y,z)=yu_k(y,z)
	\end{equation}
	and the identity (\cite[Identity (2.7)]{ribenboim})
	\begin{equation}\label{eq.hain6x5}
	v_k(y,z) = u_{k + 1}(y,z) - zu_{k - 1}(y,z),
	\end{equation}
	we have
	\[
	rv_k (y,z) + (xy - 2r)u_{k + 1} (y,z) = y\big( {xu_{k + 1} (y,z) - ru_k (y,z)} \big)
	\]
	and hence \eqref{eq.aklr76c}. The proof of \eqref{eq.iq7k21o} is similar. We use
	\begin{equation}\label{eq.m94zp3w}
	v_{k+1}(y,z) + zv_{k-1}(y,z)=yv_k(y,z)
	\end{equation}
	and the identity
	\begin{equation}\label{eq.rmgib02}
	(y^2 - 4z)u_k(y,z) = v_{k + 1}(y,z) - zv_{k - 1}(y,z).
	\end{equation}
\end{proof}
\begin{theorem}\label{thm_Luc}
	For all $x$, $y$, $r\in\mathbb{C}$ and $j$, $n\in\mathbb Z$,
	\begin{equation}\label{eq.teo1nc4}
	\sum_{k=j}^n r^{n - k} x^k \big( {rv_k (y,z) + (xy - 2r)u_{k + 1} (y,z)} \big) 
	= x^{n + 1} yu_{n + 1} (y,z) - r^{n - j + 1} x^j y u_j (y,z),
	\end{equation}
	\begin{equation}\label{eq.y7m6w1x}
	\begin{split}
	\sum_{k = j}^n r^{n - k} x^k  & \big( {(y^2 - 4z)ru_k (y,z) + (xy - 2r)v_{k + 1} (y,z)} \big)\\ 
	&\,\,\,\,\qquad\qquad\qquad\qquad\qquad\quad\qquad= x^{n + 1} yv_{n + 1} (y,z)- r^{n - j + 1} x^j y v_j (y,z). 
	\end{split}
	\end{equation}
\end{theorem}
\begin{proof}
	To prove \eqref{eq.teo1nc4}, sum \eqref{eq.aklr76c} from $j$ to $n$, making use of \eqref{tele_id1}. The proof of \eqref{eq.y7m6w1x} is similar.
\end{proof}
\begin{remark}
	Equation \eqref{eq.teo1nc4} contains one result of Abd-Elhameed and Zeya\-da \cite[Theorem 1]{ameed_zeyada18} 
	as a special case at $r=1$ and $j=0$. Identity \eqref{eq.y7m6w1x} seems to be new.
\end{remark}
\begin{lemma}
	For all $x$, $y$, $r\in \mathbb{C}$ and $k$, $n\in\mathbb Z$, we have
	\begin{equation}\label{eq.o1cpjv7}
	\begin{split}
	x^{n - k} r^k & \left( {rv_{k + 1}(y,z) + (xy + 2rz)u_k(y,z) } \right)\\[4pt]
	& \qquad\qquad\qquad\qquad\qquad = y\left( {r^{k + 1} x^{n - k} u_{k + 1}(y,z) + r^k x^{n - k + 1} u_k(y,z) } \right)\!,
	\end{split}
	\end{equation}
	\begin{equation}\label{eq.k3qycpa}
	\begin{split}
	x^{n - k} r^k & \big( (y^2 - 4z)ru_{k + 1}(y,z) + (xy + 2rz)v_k(y,z)  \big)\\[4pt] 
	&\qquad\qquad\qquad\quad\quad = y\left( {r^{k + 1} x^{n - k} v_{k + 1}(y,z) + r^k x^{n - k + 1} v_k(y,z) } \right)\!.
	\end{split}
	\end{equation}
\end{lemma}
\begin{proof}
	Eliminating $u_{k + 1}(y,z)$ between \eqref{eq.s4dauqs} and \eqref{eq.hain6x5} gives
	$
	v_k(y,z) \!=\! yu_k(y,z) - 2zu_{k - 1}(y,z)
	$
	and hence
	\[
	rv_k (y,z) + (xy + 2rz)u_{k - 1} (y,z) = ry u_k (y,z) + xyu_{k - 1} (y,z),
	\]
	from which \eqref{eq.o1cpjv7} follows. 
	
	The proof of \eqref{eq.k3qycpa} is similar; 
	eliminate $v_{k + 1}(y,z)$ between \eqref{eq.m94zp3w} and \eqref{eq.rmgib02}.
\end{proof}
\begin{theorem}\label{thm_Luc2}
	For all $x$, $y$, $r\in\mathbb{C}$ and $j$, $n\in\mathbb Z$,
	\begin{equation}\label{eq.d1bkssp}
	\begin{split}
	\sum_{k = j}^n ( - 1)^k & x^{n - k} r^k \big( {rv_{k + 1} (y,z) + (xy + 2rz)u_k (y,z)} \big) \\
	&\qquad\qquad\qquad= ( - 1)^n yr^{n + 1} u_{n + 1}(y,z)  + ( - 1)^j r^j yx^{n - j + 1} u_j(y,z),
	\end{split}
	\end{equation}
	\begin{equation}\label{eq.rht5dix}
	\begin{split}
	\sum_{k = j}^n ( - 1)^k &x^{n - k} r^k \left( {(y^2 - 4z)ru_{k + 1} (y,z) + (xy + 2rz)v_k (y,z)} \right) \\
	&\qquad\qquad\qquad= ( - 1)^n yr^{n + 1} v_{n + 1}(y,z)  + ( - 1)^j r^j yx^{n - j + 1} v_j(y,z) .
	\end{split}
	\end{equation}
\end{theorem}
\begin{proof}
	Sum each of \eqref{eq.o1cpjv7} and \eqref{eq.k3qycpa} from $j$ to $n$, making use of \eqref{tele_id2} 
	since the right-hand side telescopes in each case.
\end{proof}
\begin{remark}
	Identity \eqref{eq.d1bkssp} is an additional polynomial generalization of \eqref{id_dpl}. 
	Identity \eqref{eq.rht5dix} is its companion and also presumably new.
\end{remark}

\section{Discussion}

This section is devoted to discuss in detail some special cases of Theorems \ref{thm_Luc} and \ref{thm_Luc2}, 
hereby paying special attention to Fibonacci and Lucas polynomials $F_n(x)$ and $L_n(x)$, respectively. 
The polynomial identities presented in Proposition \ref{main_prop_1} below are also special cases of 
Theorem~1 in \cite{Chung}. We think, however, that they are worth to be discussed.
In the course of discussion we will rediscover some identities from \cite{ameed_zeyada18} 
but also present new relations that we did not find in the literature and deserve recognition. 

First we note that $$F_n(x)=u_n(x,-1),\qquad L_n(x)=v_n(x,-1)$$ and $F_n(1)=F_n$, $L_n(1)=L_n$. Also,
$F_n(2)=P_n$ and $L_n(2)=Q_n$, where $P_n$ and $Q_n$ denote the Pell and Pell--Lucas numbers, respectively. 
Other special values of these polynomials will be used later. 
\begin{proposition}\label{main_prop_1}
	For all $x,y\in\mathbb{C}$ and $n\in\mathbb Z$, we have
	\begin{align}\label{main_Fib1}
	\sum_{k=0}^n x^k \big(L_k(y) + (xy - 2)F_{k+1}(y)\big) &= x^{n+1} y F_{n+1}(y),\\
	\label{main_Luc1}
	\sum_{k=0}^n x^k \big((y^2+4)F_k(y) + (xy - 2)L_{k+1}(y)\big) & = y \big (x^{n+1} L_{n+1}(y) - 2 \big ).
	\end{align}
\end{proposition}
\begin{proof}
	Set $r=1$ and $j=0$ in \eqref{eq.teo1nc4} and \eqref{eq.y7m6w1x}.
\end{proof}

Proposition \ref{main_prop_1} provides two polynomial extensions of the identity \eqref{id_edgar}. 
Both identities connect interestingly Fibonacci and Lucas polynomials. Sin\-ce $$L_n(y) = F_{n-1}(y) + F_{n+1}(y),\qquad (y^2+4)F_n(y) = L_{n-1}(y) + L_{n+1}(y)$$ 
they may also be disconnected:
\begin{equation*}
\sum_{k=0}^n x^k \big (xF_{k+1}(y) - F_{k}(y) \big ) = x^{n+1} F_{n+1}(y)
\end{equation*} 
and
\begin{equation*}
\sum_{k=0}^n x^k \big (xL_{k+1}(y) - L_{k}(y)\big ) = x^{n+1} L_{n+1}(y) - xy.
\end{equation*}
\begin{corollary}
	For all $x\in\mathbb{C}$ and $n\in\mathbb Z$,
	\begin{align}\label{not_new1}
	\sum_{k=0}^n x^k \big(L_k + (x-2)F_{k+1}\big) &= x^{n+1} F_{n+1},\\
	\sum_{k=0}^n x^k \big(5 F_k + (x-2)L_{k+1}\big) &= x^{n+1} L_{n+1} - 2,\nonumber\\
	\label{not_new2}
	\sum_{k=0}^n x^k \big(Q_k + 2(x-1)P_{k+1}\big) &= 2 x^{n+1} P_{n+1},\\
	\sum_{k=0}^n x^k \big(4 P_k + (x-1)Q_{k+1}\big) &= x^{n+1} Q_{n+1} - 2.\nonumber
	\end{align}
\end{corollary}
\begin{proof}
	Set $y=1$ and $y=2$, in turn, in \eqref{main_Fib1} and \eqref{main_Luc1}. 
\end{proof}
We point out that \eqref{not_new1} and \eqref{not_new2} are not new. These are Equations (11) and (15) in \cite{ameed_zeyada18}. 
\begin{corollary}\label{cor_mark}
	For all $x\in\mathbb{C}$ and $n\in\mathbb Z$,
	\begin{align*}
	&\quad\sum_{k=0}^n x^k \big(L_{3k} + (2x-1)F_{3k+3}\big) = 2 x^{n+1} F_{3n+3},\\
	&\sum_{k=0}^n x^k \big(5 F_{3k} + (2x-1)L_{3k+3}\big) = 2 x^{n+1} L_{3n+3} - 4.
	\end{align*}
\end{corollary}
\begin{proof}
	Set $y=4$ in \eqref{main_Fib1} and \eqref{main_Luc1}, respectively, and use the evaluations $F_n(4)=\frac12F_{3n}$ as well as $L_n(4)=L_{3n}$.
\end{proof}
\begin{corollary}
	For all $x\in\mathbb{C}$ and $n\in\mathbb Z$,
	\begin{align*}
	&\quad\qquad\qquad\sum_{k=0}^n 2^k x^{n-k} L_k(x) = 2^{n+1} F_{n+1}(x),\\
	&\sum_{k=0}^n 2^k x^{n-k} F_k(x) = \frac{2}{x^2+4}\Big ( 2^{n} L_{n+1}(x)-x^{n+1}\Big),\quad x\neq \pm 2i.
	\end{align*}
\end{corollary}
\begin{proof}
	Set $x=\frac{2}{y}$ in \eqref{main_Fib1} and \eqref{main_Luc1}, respectively. 
\end{proof}

The results in Corollaries \ref{main_cor_2}\,--\,\ref{cor.x2hi58v} are identities involving only even-inde\-xed Fibonacci and Lucas polynomials.
\begin{corollary}\label{main_cor_2}
	For all $x,y\in\mathbb{C}$ and $n\in\mathbb Z$, we have
	\begin{equation}\label{main_Fib2}
	\sum_{k=0}^n x^k \big (y L_{2k}(y) + (x(y^2+2) - 2)F_{2k+2}(y) \big ) = x^{n+1} (y^2+2) {F_{2n+2}(y)},
	\end{equation}
	\begin{equation}\label{main_Luc2}
	\sum_{k=0}^n x^k \big ( y(y^2+4) F_{2k}(y) + \big(x(y^2+2) - 2\big)L_{2k+2}(y) \big ) = (y^2+2)\big ( x^{n+1} L_{2n+2}(y) - 2\big).
	\end{equation}
\end{corollary}
\begin{proof} We only prove \eqref{main_Fib2}. Use Proposition \ref{main_prop_1} and relations
	$$	L_n\big(i(x^2+2)\big) = i^n L_{2n}(x),\qquad F_n \big(i(x^2+2)\big) = i^{n-1} \frac{F_{2n}(x)}{x},$$
	where $i=\sqrt{-1}$. When simplifying replace $x$ by $\frac{x}{i}$. The other proof is similar.
\end{proof}
\begin{corollary}
	For all $x\in\mathbb{C}$ and $n\in\mathbb Z$,
	\begin{align*}\label{Fib_even}
	\sum_{k=0}^n x^k \big(L_{2k} + (3x-2)F_{2k+2}\big) &= 3 x^{n+1} F_{2n+2},\\
	\sum_{k=0}^n x^k \big(5 F_{2k} + (3x-2)L_{2k+2}\big) &= 3\big (x^{n+1} L_{2n+2} - 2 \big ),\\
	\sum_{k=0}^n x^k \big(Q_{2k} + (3x-1)P_{2k+2}\big) &= 3 x^{n+1} P_{2n+2},\\
	\sum_{k=0}^n x^k \big(8 P_{2k} + (3x-1)Q_{2k+2}\big) &= 3\big ( x^{n+1} Q_{2n+2} - 2 \big).
	\end{align*}
\end{corollary}
\begin{proof}
	Set $y=1$ and $y=2$ in \eqref{main_Fib2} and \eqref{main_Luc2}, respectively. 
\end{proof}
\begin{corollary}
	For all $x\in\mathbb{C}$ and $n\in\mathbb Z$,
	\begin{align*}
	&\quad\,\,\, \sum_{k=0}^n x^k \big(3L_{4k} + (7x-2)F_{4k+4}\big) = {7} x^{n+1} F_{4n+4},\\
	&\sum_{k=0}^n x^k \big(15 F_{4k} + (7x-2)L_{4k+4}\big) = 7 (x^{n+1} L_{4n+4} - 2).
	\end{align*}
\end{corollary}
\begin{proof}
	Set $y=\sqrt{5}$ in \eqref{main_Fib2} and \eqref{main_Luc2}, respectively, and use the evaluations 
	$F_{2n}(\sqrt{5})=\frac{\sqrt{5}}{3} F_{4n}$ as well as $L_{2n}(\sqrt{5})=L_{4n}$.
\end{proof}
\begin{corollary}\label{cor.x2hi58v}
	For all $x\in\mathbb{C}\setminus\!\{0\}$ and $n\in\mathbb Z$,
	\begin{equation*}
	\sum_{k=0}^n \Big (\frac{x^2+2}{2}\Big )^{n-k} L_{2k}(x) = \frac{2}{x} F_{2n+2}(x), 
	\end{equation*}
	\begin{equation*}
	\sum_{k=0}^n \Big (\frac{x^2+2}{2}\Big )^{n-k} F_{2k}(x) = \frac{2}{x(x^2+4)} \bigg (L_{2n+2}(x)-2 \Big( \frac{x^2+2}{2}\Big )^{n+1}\bigg ).
	\end{equation*}
\end{corollary}
\begin{proof}
	Set $x=\frac{2}{y^2+2}$ in \eqref{main_Fib2} and \eqref{main_Luc2}, respectively, and replace $y$ by~$x$. 
\end{proof}

The next identities provide still other interesting relations that nicely generalize and complement Edgar--Bhatnagar identity \eqref{id_edgar}.
\begin{proposition}\label{FmLm_ids}
	For all $x\in\mathbb{C}$ and integers $m$ and $n$, we have
	\begin{align}\label{main_Fib3}
	&\qquad \sum_{k=0}^n x^k \big (F_mL_{mk} + (x L_m - 2){F_{m(k+1)}}\big ) = x^{n+1}  {F_{m(n+1)}}L_m,\\
	&\sum_{k=0}^n x^k \big ( 5 F_m F_{mk} + (x L_m - 2)L_{m(k+1)} \big ) =  \big ( x^{n+1} L_{m(n+1)} - 2 \big)L_m.\nonumber
	\end{align}
\end{proposition}
\begin{proof}
	We only prove \eqref{main_Fib3}. Let $m$ be odd. Then by straightforward calculation
	$$	L_k(L_m) = L_{mk} \text{\quad and\quad} F_k(L_m) = \frac{F_{mk}}{F_m},$$ 
	where we used the identity $L_n^2 = 5F_n^2 + (-1)^n 4$. This immediately gives the first identity, using \eqref{main_Fib1}. 
	Let $m$ be even now. Then we work with $i L_m$ as an argument and get
	$$L_k(i L_m) = i^k L_{mk} \text{\quad    and\quad} F_k(i L_m) = i^{k-1} \frac{F_{mk}}{F_m}.$$ 
	Inserting these results into \eqref{main_Fib1} and finally replacing $x$ by $\frac{x}{i}$ completes the proof.
\end{proof}

We proceed with a (short) discussion of Theorem \ref{thm_Luc2}. Making the choices $z=-1$, $r=1$ and $j=0$ it becomes 
the next identities for Fibonacci and Lucas polynomials, which can also be obtained from Theorem 1 in \cite{Chung}.
\begin{proposition}
	For all $x$, $y\in\mathbb{C}$ and $n\in\mathbb Z$, we have
	\begin{equation}\label{main_Fib4}
	\sum_{k=0}^n (-1)^{n-k} x^{n-k} \big(L_{k+1}(y) + (xy - 2)F_{k}(y)\big) = y F_{n+1}(y),
	\end{equation}
	\begin{equation}\label{main_Luc4}
	\sum_{k=0}^n (-1)^{n-k} x^{n-k} \big((y^2+4)F_{k+1}(y) + (xy - 2)L_{k}(y)\big)
	=  y L_{n+1}(y) + 2(-1)^{n}y x^{n+1}.
	\end{equation}
\end{proposition}
\begin{corollary}
	For all $x\in\mathbb{C}$ and $n\in\mathbb Z$,
	\begin{align*}
	\sum_{k=0}^n (-1)^{n-k} x^{n-k} \big(L_{k+1} + (x-2)F_{k}\big) &=  F_{n+1}, \\
	\sum_{k=0}^n (-1)^{n-k} x^{n-k} \big(5 F_{k+1} + (x-2)L_{k}\big) &=  L_{n+1} + 2 (-1)^{n}x^{n+1}, \\
	\sum_{k=0}^n (-1)^{n-k} x^{n-k} \big(Q_{k+1} + 2(x-1)P_{k}\big) &= 2  P_{n+1},  \\
	\sum_{k=0}^n (-1)^{n-k} x^{n-k} \big(4 P_{k+1} + (x-1)Q_{k}\big) &= Q_{n+1} + 2 (-1)^{n}x^{n+1}. 
	\end{align*}
\end{corollary}
\begin{proof}
	Set $y=1$ and $y=2$ in \eqref{main_Fib4} and \eqref{main_Luc4}, respectively. 
\end{proof}

The first identity in the corollary above was derived by Martiniak and Prodinger \cite{Martinjak}. 

We are not interested in repeating the analogous results to Corollaries \ref{cor_mark}--\ref{cor.x2hi58v}.
We mention, however, that the symmetry gets lost when passing from Proposition \ref{FmLm_ids} to its alternating analogue. 
\begin{proposition}
	For all $x\in\mathbb{C}$ and integers $m$ and $n$, we have
	\begin{equation*}\label{main_Fib51}
	\sum_{k=0}^n (-1)^{n-k} x^{n-k} \big (F_m L_{m(k+1)} + (x L_m - 2){F_{mk}}\big ) = L_m {F_{m(n+1)}}, \quad m\,\, \mbox{odd,}
	\end{equation*}
	\begin{equation*}\label{main_Fib52}
	\sum_{k=0}^n x^{n-k} \big ( F_m L_{m(k+1)} - (x L_m - 2)F_{mk} \big ) = L_m {F_{m(n+1)}}, \quad m\,\, \mbox{even}.
	\end{equation*}
\vspace{2pt}
	
	Similarly, we have
	\begin{align*}\label{main_Luc51}
	\sum_{k=0}^n (-1)^{n-k} x^{n-k} \big (5 F_m F_{m(k+1)} + (x L_m - 2){L_{mk}}\big ) = L_{m}\big( L_{m(n+1)} + 2(-1)^{n} x^{n+1}\big ), \quad m\, \mbox{odd},
	\end{align*}
	\begin{equation*}\label{main_Luc52}
	\sum_{k=0}^n x^{n-k} \big ( 5 F_m F_{m(k+1)} - (x L_m - 2)L_{mk} \big ) = L_m \big (L_{m(n+1)} - 2 x^{n+1}\big ), \quad m \mbox{ even}.
	\end{equation*}
\end{proposition}

\section{Extension to Chebyshev polynomials}

Let $T_n(y)$ and $U_n(y)$ be the Chebyshev polynomials of the first and second kind defined by 
\begin{align*}
T_n(y) &= 2yT_{n - 1}(y) - T_{n - 2}(y),\quad n\ge 2,\quad T_0(y)=1,\, T_1(y)=y,\\[4pt]
U_n(y) &= 2yU_{n - 1}(y) - U_{n - 2}(y),\quad n\ge 2,\quad U_0(y)=1,\, U_1(y)=2y.
\end{align*}
Both sequences $(T_n(y))_{n\geq0}$ and $(U_n(y))_{n\geq0}$ can be extended to negative subscripts by writing
$$T_{-n}(y) = T_n(y), \qquad U_{-n}(y) = 2y U_{-(n - 1)}(y) - U_{-(n - 2)}(y).
$$
\begin{lemma}
	For all $x$, $y$, $r\in \mathbb{C}$, $r\ne 0$ and $k\in\mathbb Z$, we have
	\begin{equation}\label{Cheb_id1}
	\frac{x^k}{r^k} \big( rT_k (y) + (xy - r)U_k (y) \big) = y\left( \frac{x^{k + 1}}{r^k} U_k (y) - \frac{x^k}{r^{k - 1}}U_{k - 1} (y) \right)\!,
	\end{equation}
	\begin{equation}\label{Cheb_id2}
	\frac{x^k}{r^k} \left( (y^2 - 1)rU_{k - 2} (y) + (xy - r)T_k (y) \right) = y\left( \frac{x^{k + 1}}{r^k}T_k (y) 
	- \frac{x^k}{r^{k - 1}}T_{k - 1} (y) \right)\!. 
	\end{equation}
\end{lemma}
\begin{proof}
	In view of the recurrence relation
	$
	U_k (y) = 2yU_{k - 1} (y) - U_{k - 2} (y)
	$
	and the identity
	$
	2T_k (y) = U_k (y) - U_{k - 2} (y),
	$
	we have
	\[
	rT_k (y) + (xy - r)U_k (y) = y\big(xU_k(y) - rU_{k - 1}(y)\big),
	\]
	from which \eqref{Cheb_id1} follows. Similarly, using the recurrence relation
	\[
	T_k(y)=2yT_{k - 1}(y) - T_{k - 2}(y)
	\]
	and the identity
	$$
	2(y^2 - 1)U_{k - 2}(y)=T_k(y) - T_{k - 2}(y),
	$$
	we obtain
	\[
	(y^2 - 1)rU_{k - 2}(y) + (xy - r)T_k(y)=y\big(xT_k(y) - rT_{k - 1}(y)\big)
	\]
	from which \eqref{Cheb_id2} follows.
\end{proof}
\begin{theorem}\label{Cheb_main}
	For all $x$, $y$, $r\in\mathbb{C}$ and $j$, $n\in\mathbb Z$, we have
	\begin{align*}
	&\qquad\sum_{k = j}^n {r^{n - k} x^k \big( {rT_k (y) + (xy - r)U_k (y)} \big)} = x^{n + 1} yU_n (y) - r^{n - j + 1} x^j yU_{j - 1} (y),\\
	&\sum_{k = j}^n r^{n - k} x^k \left( {(y^2 - 1)rU_{k - 2} (y) + (xy - r)T_k (y)} \right) = x^{n + 1} yT_n (y) - r^{n - j + 1} x^j yT_{j - 1} (y).
	\end{align*}
	In particular, we have
	\begin{equation*}
	\sum_{k = 0}^n {x^k \big( {T_k (y) + (xy - 1)U_k (y)} \big)} = x^{n + 1} yU_n (y),
	\end{equation*}
	as well as
	\begin{equation*}
	\sum_{k = 0}^n {x^k \big( {(y^2 - 1)U_{k - 2} (y) + (xy - 1)T_k (y)} \big)} = x^{n + 1} yT_n (y) - y^2.
	\end{equation*}
\end{theorem}
\begin{proof}
	Sum each of \eqref{Cheb_id1} and \eqref{Cheb_id2}, using \eqref{tele_id1}.
\end{proof}
\begin{lemma}\label{Cheb_values} 
	For all  $x \in \mathbb{C}$, and all $n\geq 0$, we have
	\begin{align*}
	&U_n \Big( \frac{x^2+2}{2} \Big) = (-1)^{n+1} \frac{i}{x} U_{2n+1}\Big ( \frac{ix}{2}\Big ),\quad x\ne0,\\ 
	&\qquad\qquad  T_n \Big( \frac{x^2+2}{2} \Big) = (-1)^n T_{2n}\Big (\frac{ix}{2}\Big ).
	\end{align*}
\end{lemma}
\begin{proof}
	Both identities can be proved by induction on $n$.
\end{proof}
\begin{corollary}
	For all nonzero $x,y \in \mathbb{C}$, and all $n\geq 0$, we have
	\begin{align*}
	&\sum_{k=0}^n (-1)^{n-k} x^k \Big ( 2yT_{2k}(y) + \big(x(1-2y^2)-1\big) U_{2k+1}(y)\Big )\! = x^{n+1}(1-2y^2) U_{2n+1}(y),\\
	&\sum_{k=0}^n (-1)^{n-k} x^k \Big ( 2y (y^2-1) U_{2k-3}(y) + \big(x(1-2y^2)-1\big) T_{2k}(y)\Big ) \\
	&\qquad\qquad\qquad\qquad\qquad\qquad\qquad\qquad = x^{n+1} (1-2y^2) T_{2n}(y) -  (-1)^n(1-2y^2)^2
	\end{align*}
	with $U_{-3}(y)=-2y$, $U_{-2}(y) = -1$ and $U_{-1}(y) = 0$.
\end{corollary}
\begin{proof}
	Combine Theorem \ref{Cheb_main} with Lemma \ref{Cheb_values} and simplify.
\end{proof}

\section{Conclusion}

In these notes we have provided various complements to Fibonacci-Lucas identities first proved by Edgar, Sury, and Bhatnagar. 
Starting with additional polynomial generalizations, we have also presented a detailed discussion of charming identities 
for Fibonacci and Lucas polynomials as special cases. Finally, we have proved the analogues for Chebyshev polynomials. 
We conclude noting two observations. 
First, concerning the results for Fibonacci and Lucas polynomials we remark that our results can be 
combined with the beautiful identities derived by Seiffert \cite{Seiffert1,Seiffert2,Seiffert3} 
to get some nontrivial sum relations. This is left as a possible future research project.
Second, as it was shown here, the identities introduced and discussed are not limited 
to Fibonacci and Lucas, and Chebyshev polynomials. Analogues are possible for Pell polynomials, 
Jacobsthal polynomials and others. By way of a final illustration, we state the identities in Theorem \ref{thm_Luc} 
for the Jacobsthal and Jacobsthal--Lucas sequences $(J_n)=(u_n(1,-2))$ and $(j_n)=(v_n(1,-2))$:

For all $x$, $r\in\mathbb{C}$ and $s$, 
$n\in\mathbb Z$, we have
\begin{align*}
\sum_{k = s}^n r^{n - k} x^k \big(rj_k  + (x - 2r)J_{k + 1} \big)  = x^{n + 1} J_{n + 1}  - r^{n - s + 1} x^s J_s,\\ 
\sum_{k = s}^n r^{n - k} x^k \big(9rJ_k  + (x - 2r)j_{k + 1}\big )  = x^{n + 1} j_{n + 1}  - r^{n - s + 1} x^s j_s. 
\end{align*}
Interesting identities can be drawn from these relations. We leave them for the interested readers.

\end{document}